\documentclass[11pt]{amsart}
\newtheorem{theorem}{Theorem}[section]

\newtheorem{corollary}[theorem]{Corollary}

\newtheorem{proposition}[theorem]{Proposition}
\newtheorem{remark}[theorem]{Remark}
\newtheorem{definition}[theorem]{Definition}
\newtheorem{example}[theorem]{Example}

\usepackage[vmargin=2cm, hmargin=2cm]{geometry}

\usepackage{fourier}

\usepackage{hyperref}
\usepackage{enumerate}
\usepackage{amsmath}
\usepackage{amssymb}

\newcommand{\uno}{\mathbf{1}}

\newcommand{\ld}{\leq_\diamond}

\def\RR{{\mathbb{R}}}
\def\CC{{\mathbb{C}}}

\def\soc{{\rm Soc}}

\begin{document}
\title[Maps preserving the diamond partial order]{Maps preserving the diamond partial order}

\author{M. Burgos}

\address{Campus de Jerez, Facultad de Ciencias Sociales y de la Comunicaci\'{o}n Av. de la Universidad s/n, 11405 Jerez, C\'{a}diz, Spain}

\email{maria.burgos@uca.es}

\author{A. C. M\'{a}rquez-Garc\'{i}a}

\address{ Departamento \'{A}lgebra y An\'{a}lisis Matem\'{a}tico,
Universidad de Almer\'{i}a, 04120 Almer\'{i}a, Spain}

\email{acmarquez@ual.es}

\author{A. Morales-Campoy}

\address{Departamento de \'{A}lgebra y An\'{a}lisis Matem\'{a}tico,
Universidad de Almer\'{i}a, 04120 Almer\'{i}a, Spain}

\email{amorales@ual.es}

\thanks{Authors partially supported by the Spanish Ministry of Economy and Competitiveness project no. MTM2014-58984-P and Junta de Andaluc\'{\i}a grants FQM375, FQM194. The second author is also supported by a Plan Propio de Investigaci\'{o}n grant from University of Almer\'{i}a. The authors thank A. Peralta for useful comments and hospitality during their visit to the Departamento de An\'{a}lisis Matem\'{a}tico de la Universidad de Granada.}

\begin{abstract} The present paper is devoted to the study of the diamond partial order in general $C^*$-algebras and the description of linear maps preserving this partial order.
\end{abstract}


\date{}

\maketitle
 \thispagestyle{empty}

\section{Introduction}
Let $A$ be a (complex) Banach algebra. An element $a$ in $A$
is (\emph{von Neumann}) \emph{regular} if it has a
generalized inverse, that is, if there exists $b$ in $A$ such that
$a=aba$ ($b$ is an \emph{inner inverse} of $a$) and $b=bab$ ($b$ is
an \emph{outer inverse} of $a$). The generalized inverse of a regular element $a$ is not unique. Observe also that the first equality
$a=aba$ is a necessary and sufficient condition for $a$ to be
regular, and that, if $a$ has generalized inverse $b$, then $p=ab$
and $q=ba$ are idempotents in $A$ with $aA=pA$ and
$Aa=Aq$. We denote by $A^\bullet$ the set of idempotent elements in $A$ and by $A^\wedge$ the set of all regular elements of $A$.

The unique generalized inverse of $a$ that commutes with $a$ is called the \emph{group inverse} of $a$, whenever it exists. In this case $a$ is said to be \emph{group invertible} and its group inverse is denoted by $a^\sharp$. The set of all group invertible elements of $A$  is denoted by $A^\sharp$.

\smallskip

For an element $a$ in $A$ let us consider the left and right
multiplication operators $L_a :x\mapsto ax$ and $R_a:x\mapsto xa$,
respectively. If $a$ is regular, then so are $L_a$ and $R_a$, and
thus their ranges $aA=L_a(A)$ and $Aa=R_a(A)$ are both
closed.
\smallskip

Regular elements in unital C$^*$-algebras have been
studied by Harte and Mbekhta. The main result in \cite{HarMb92} states  that an element $a$ in a
C$^*$-algebra $A$ is regular if and only if $aA$ is closed.

Given $a$ and $b$ in $A$, $b$ is said to be a
\emph{Moore-Penrose inverse} of $a$ if $b$ is a generalized inverse
of $a$ and the associated idempotents $ab$ and $ba$ are
selfadjoint (i.e., projections). It is known that every regular element $a$ in $A$ has
a unique Moore-Penrose inverse that will be denoted by $a^\dag$ (\cite[Theorem 6]{HarMb92}).
Therefore, the Moore-Penrose inverse of a regular element $a\in A$ is the unique element that satisfy the following equations:
$$aa^\dagger a=a,\quad a^\dagger a a^\dagger=a^\dagger, \quad (aa^\dagger)^*=aa^\dagger, \quad (a^\dagger a)^*=a^\dagger a.$$
In what follows let us denote by $\textrm{Proj}(A)$ the set of projections of $A$.

Generalized inverses are used in the study of partial orders on matrices, operator algebras and abstract rings.  Let $M_n(\CC)$ be the algebra of all $n\times n$ complex matrices. On  $M_n(\CC)$ there are many classical partial orders (see \cite{BakHa90} , \cite{Drazin78}, \cite{Hart80}, \cite{HarSty86}, \cite{Mi87}, \cite{Mi91}, \cite{MiBhiMa}).
The \emph{star partial order} on $M_n(\CC)$ was introduced by Drazin  in \cite{Drazin78}, as follows:
$$ A\leq_{*}B \qquad \mbox{if and only if}\qquad A^*A=A^*B \,\,  \mbox{and } \,\, AA^*=BA^*,$$
where as usual $A^*$ denotes the conjugate transpose of $A$. He showed that $ A\leq_{*}B$ if and only if $A^\dag A=A^\dag B $ and $ AA^\dag=BA^\dag .$
Baksalary and Mitra introduced in \cite{BakMi91} the \emph{left-star} and  \emph{right-star} partial order on $M_n(\CC)$ , as
$$ A*\leq B \qquad \mbox{if and only if}\qquad A^*A=A^*B \, \, \mbox{and } \,\,  \textrm{Im} A\subseteq \textrm{Im}B,$$
and
$$ A\leq * B \qquad \mbox{if and only if}\qquad A A^*=B A^* \,\,  \mbox{and }\,\,  \textrm{Im}A^*\subseteq \textrm{Im}B^*,$$ respectively.
Besides, $ A\leq_{*}B$ if and only if $ A*\leq B$ and $ A\leq * B$.

Hartwig \cite{Hart80} introduced the \emph{rank substractivity order}, usually known as the \emph{minus partial order} on $M_n(\CC)$:
$$ A\leq^ {-} B \qquad \mbox{if and only if}\qquad \textrm{rank}(B-A)= \textrm{rank}(B)- \textrm{rank}(A).$$ It is proved that
$$ A\leq^{-}B \qquad \mbox{if and only if}\qquad A^-A=A^-B \,\,  \mbox{and }\, \, AA^-=BA^-,$$ where $A^-$ denotes an inner inverse of $A$.
Later, Mitra used in \cite{Mi87} the group inverse of a matrix to define the \emph{sharp order} on group invertible matrices:
$$ A\leq_{\sharp}B \qquad \mbox{if and only if}\qquad A^\sharp A=A^\sharp B \,\,  \mbox{and }\, \, AA^\sharp=BA^\sharp.$$

Let $H$ be an infinite-dimensional complex Hilbert space, and $B(H)$ the C$^*$-algebra af all bounded linear operators on $H$. \v{S}emrl  extended in \cite{Semrl10} the minus partial order from  $M_n(\CC)$ to $B(H)$ finding an appropriate equivalent definition of the minus partial order on  $M_n(\CC)$ which does not involve inner inverses:  for $A,B\in B(H)$, $A\preceq B$ if and only if there exists idempotent operators $P,Q\in B(H)$ such that
$$R(P)=\overline{R(A)},\quad N(A)=N(Q), \quad PA=PB, \quad AQ=BQ.$$
\v{S}emrl proved that the relation  ``$\preceq $'' is a partial order in $B(H)$ extending the minus partial order of matrices. Recently Djordjevi\'c, Raki\'c and Marovt (\cite{DjoRaMa13}) generalized \v{S}emrl's definition to the environment of Rickart rings and generalized some well known results.
%$A\leq^{-} B$ if there exist two idempotent operators $P,Q\in B(H)$ such that $\textrm{Im}(P)=\overline{\textrm{Im}(A)}

%Following \v{S}emrl's approach, Dolinar and Marovt extended in \cite{DoMa11} the star partial order  from  $M_n(\CC)$ to $B(H)$. From \cite[Theorem 5]{DoMa11}, for $T,S\in B(H)$, $T\leq_{*}S$ if and only if, there exist two selfadjoint idempotent operators $P,Q \in B(H)$, such that $\textrm{Im}(P)=\overline{\textrm{Im}(T)}$, $\textrm{Ker}(Q)=\textrm{Ker}(T)$, $PT=PS$, and  $TQ=SQ$.

One of the most active and fertile research area in Linear Algebra, Operator Theory and Functional Analysis, are the ``linear preserver problems''. These problems concern the characterization of linear maps between algebras that, roughly speaking,  leave certain functions, subsets, relations, properties... invariant. The goal is to find the form of these maps. (See for instance \cite{GuLiSe, Molnar} and the references therein.)

 In   \cite{Semrl10},  \v{S}emrl  studied bijective maps preserving the minus partial order. For an infinite-dimensional complex Hilbert space $H$, a mapping  $\phi:B(H)\to B(H)$  preserves the minus partial order if  $ A\preceq B$ implies that  $ \phi(A)\preceq\phi(B)$.
The map  $\phi:B(H)\to B(H)$  preserves the minus order in both directions whenever  $ A\preceq B$ if and only if $ \phi(A)\preceq\phi(B)$. He proved that a bijective map  $\phi:B(H)\to B(H)$  preserving the minus partial order in both directions is either of the form $\phi(A)=TAS$ or $\phi(A)=TA^*S$, for some invertible operators $T$ and $S$ (both linear in the first case and both conjugate linear in the second one).

In \cite{Gut07} Guterman studied additive maps preserving the star, left-star and right-star orders between real and complex matrix algebras. An additive map $\phi:M_n(\CC)\to M_n(\CC)$ preserves the star partial order if $ A\leq_{*}B$ implies that  $ \phi(A)\leq_{*}\phi(B)$. Additive maps preserving the left-star and right-star partial order are defined in a similar way. Recently, the authors of \cite{DoGuma14} bring some results from \cite{Gut07} concerning left and right star partial orders to the infinite-dimensional case, following some techniques from  \cite{Semrl10}.
%They show that every bijective additive map $\phi:B(H)\to B(H)$ preserving the left-star partial order in both directions has the form $\phi(A)=UAS$ for all $A\in B(H)$, where $U$ is a unitary operator and $S$ is bijective (note that both $U$ and $S$ can be linear or conjugate linear). Analogous conclusions are obtained also for the right-partial order.

Linear maps preserving the sharp partial order and the star partial order in semisimple Banach algebras and C$^*$-algebras are studied in  \cite{BuMaPa15}. It is introduced a new relation $(R1)$ which extends the sharp relation to the full algebra:
$$(R1)\qquad a\leq_s b\quad\mbox{ if and only if there exists}\, p\in A^\bullet \, \mbox{ such that }\, a=pb=bp.$$
It is shown that a bijective linear map preserving the sharp partial order (respectively, the relation $(R1)$) from a unital semisimple Banach algebra with essential socle into a Banach algebra is a Jordan isomorphism multiplied by a invertible central element (\cite[Theorems 2.7, 2.16]{BuMaPa15}). The authors also consider the relation:
$$ (R2) \qquad a\leq b \quad \mbox{if and only if}\quad a=pb=bq\, \mbox{ for some }\, p,q\in \textrm{Proj}(A),$$
which is equivalent to the star partial order for Rickart C$^*$-algebras. Every bijective linear map preserving the relation $(R2)$, from a unital $C^*$-algebra with large socle into a C$^*$-algebra is a  Jordan *-homomorphism multiplied by an invertible element (\cite[Corollary 3.7]{BuMaPa15})
If $A$ is a real rank zero $C^*$-algebra, $B$ is a C$^*$-algebra and $T:A\to B$ is a bounded linear map preserving the relation $(R2)$, then $T$ is a linear map preserving orthogonality (\cite[Theorem 3.10]{BuMaPa15}), and thus it is an appropriate multiple of a Jordan $^*$-homomorphism (\cite[Theorem 17 and Corollary 18]{Orth08}).

Motivated by the definition of  Djordjevi\'c, Raki\'c and Marovt (\cite{DjoRaMa13})  of the minus partial order in Rickart rings, the authors of the present paper, consider in \cite{BuMaMo-pr15} the minus partial order in a unital ring $A$: for an element $a\in A$, let  $\textrm{ann}_r (a)= \{x\in A\colon ax=0\}$ and $\textrm{ann}_l (a)=  \{x\in A\colon xa=0\}$, the right and left annihilator of $a$, respectively.
We say that $a\leq^- b$ if there exist $p,q\in A^\bullet $ such that  $\textrm{ann}_l(a)=\textrm{ann}_l(p)$, $\textrm{ann}_r(a)=\textrm{ann}_r(q)$, $pa=pb$ and $aq=bq$.
It is shown that this is a partial order when restricted to the set of all regular elements in a semiprime ring. Several well known results for matrices and bounded linear operators on Banach spaces are also obtained. Moreover,  when $A$ and $B$ are unital semisimple Banach algebras with essential socle, it is proved that every bijective linear mapping $\Phi:A \to B$ such that $\Phi(A^\wedge)= B^\wedge$ and $a\leq^-b \Leftrightarrow \Phi(a)\leq^- \Phi(b)$ for every $a,b\in A^\wedge$, is a Jordan isomorphism multiplied by an invertible element.

The paper is organized as follows. In Section 2 we recall the definition of the diamond partial order introduced by Lebtahi, Patr\'{i}cio and Thome in \cite{LebPaTh13} for regular *-rings. We show that this is a partial order in every $C^*$-algebra and describe some distinguished elements with respect to this relation such as the maximal and minimal elements (Proposition \ref{invprdiamond} and Proposition \ref{minimals}, respectively). We also characterize projections and multiples of isometries and co-isometries by means of the diamond partial order. These results will be applied in Section 3 where we study linear maps between $C^*$-algebras preserving the diamond partial order. Every Jordan $^*$-homomorphism preserves the diamond partial order on regular elements (Proposition \ref{jordiam}). In Theorem \ref{main1} we prove that every surjective linear map $T:A\to B$ between  unital $C^*$-algebras with essential socle ($B$ is assumed to be prime), that preserves the diamond partial order in both directions, is an appropriate multiple of a  Jordan $^*$-homomorphism. We also prove  in Theorem \ref{d-rro} that, if $A$ is a real rank zero $C^*$-algebra, $B$ is a $C^*$-algebra and $T:A\to B$ is a bounded linear map preserving the diamond partial order, then $T$ is a   Jordan $^*$-homomorphism (respectiveley, a  Jordan $^*$-homomorphism multiplied by a unitary element) whenever $T(\textbf{1})\in \textrm{Proj}(B)$  (respectively, $T(A)\cap B^{-1}$ and $T(\textbf{1})$ is a partial isometry).
\section{Diamond partial order}
In \cite{LebPaTh13}, Lebtahi, Patr\'{i}cio and Thome introduce the diamond partial order on a $^*$-regular ring, extending a partial order defined in the matrix setting by Baksalary and Hauke in \cite{BakHa90}. Although it can be considered in a more general setting, we will focus on the framework of  C$^*$-algebras.

\begin{definition}\label{diamdef}Let $A$ be a unital C$^*$-algebra and $a,b\in A$. We say that $a\ld b$ if and only if $aA\subset bA$, $Aa\subset Ab$ and $aa^*a=ab^*a$.
\end{definition}
Let $A$ be a unital C$^*$-algebra and $a,b\in A$. We say that $a\leq_{sp} b$ if  $aA\subset bA$ and $Aa\subset Ab$. This definition is analogous to that of space pre-order on complex matrices introduced by Mitra in \cite{Mi91}. Therefore,
$$ a\ld b \quad \mbox{if and only if}\quad a\leq_{sp} b \, \mbox{ and }\, aa^*a=ab^*a.$$
The following proposition collects some algebraic properties of the relation ''$\ld$ '' that will we need in the sequel. It is implicitly proved in \cite{LebPaTh13}.
Recall that $a\leq^- b$ if there exist $p,q\in A^\bullet $ such that  $\textrm{ann}_l(a)=\textrm{ann}_l(p)$, $\textrm{ann}_r(a)=\textrm{ann}_r(q)$, $pa=pb$ and $aq=bq$. This relation defines a partial order when restricted to the set of all regular elements, and  given $a, b\in A^\wedge$, $a\leq^- b$ if and only if there exists an inner inverse, $b^-$, of $b$ such that $a=ab^- b=bb^- a=ab^-a.$ (Compare with \cite[Proposition 2.1, Corollary 2.4]{BuMaMo-pr15}).
\begin{proposition}\label{diamprop}Let $A$ be a unital C$^*$-algebra.
\begin{enumerate}
\item If $a\in A^\wedge$ and $b\in A$, $a\ld b$ if and only if $a\leq_{sp} b$ and $a^\dag b a^\dag=a^\dag$.

\item If $a\in A^\wedge$ and $b\in A$, then $a\ld b$ whenever $a\leq_{*} b$.

\item Given $a,b\in A^\wedge$, $a\ld b$ if and only if $a^\dag\leq^{-} b^\dag$.

\end{enumerate}
\end{proposition}
\begin{proof}
\emph{(1)} See \cite[Theorem 1]{LebPaTh13}.

\emph{(2)} See  \cite[Proposition 2 (a)]{LebPaTh13}.

\emph{(3)} See  \cite[Theorem 2]{LebPaTh13}.
\end{proof}
It follows from \emph{(3)} and the fact that ``$\leq^{-}$'' is a partial order on the set of all regular elements (see \cite[Corollary 2.4]{BuMaMo-pr15}) that the relation ``$\ld$ '' is a partial order on $A^\wedge$. Besides, we can state the following:
\begin{proposition}\label{diam-part} Let $A$ be a unital  C$^*$-algebra. The relation ``$\ld$ '' is a partial order on $A$.
\end{proposition}
\begin{proof}
Reflexivity of the relation ``$\ld$ '' is clear.

Let $a,b\in A$ such that $a\ld b$ and $b\ld a$. In particular, $aa^* a=ab^*a$, $bb^*b=ba^*b$, and there exist $x, y \in A$ such that $a=xb=by$.
Since $bb^*b=bb^* x^* b$, it follows by cancellation that, $b^*b=b^* x^* b=a^*b$. That is,  $b^*b=y^*b^*b$, which shows that $b^*=y^*b^*=a^*$, equivalently $a=b$. This proves that the relation ``$\ld$ '' is  anti-symmetric.

Finally, in order to prove the transitivity of ``$\ld$ '', take  $a,b,c\in A$ such that $a\ld b$ and $b\ld c$. Clearly, $a\leq_{sp} c$. Let $x,y\in A$ be such that $a=xb=by$. If follows that
$$aa^*a=ab^*a=xbb^*by=xbc^*by=ac^*a,$$
and hence $a\ld c$, as desired.
\end{proof}
In the next proposition we characterize projections in terms of the diamond partial order. We generally denote the identity element of any C$^*$-algebra by $\textbf{1}$.
\begin{proposition}\label{proj} Let $A$ be a unital $C^*$-algebra. The following conditions are equivalent:
\begin{enumerate}
\item $p\in \textrm{Proj}(A)$,
\item $p\ld \textbf{1}$ and $ \textbf{1}-p\ld  \textbf{1}$,
\item there is $q\in \textrm{Proj}(A)$, such that $p\ld q$ and $ q-p\ld  q$.
\end{enumerate}
\begin{proof}
It is clear that \emph{(1)}$\Rightarrow$\emph{(2)}$\Rightarrow$\emph{(3)}.

Assume that \emph{(3)} holds. Let $q\in \textrm{Proj}(A)$  such that $p\ld q$ and $ q-p\ld  q$. There exist $x,y \in A$ such that $p=qx=yq$, which shows that \begin{equation}\label{p1}
p=qp=pq.
\end{equation} Hence
\begin{equation}\label{p2}
pp^*p=pqp=p^2.
\end{equation} Moreover, by transitivity, since $ q-p\ld  q$, we have $ q-p\ld  \textbf{1}.$ In particular,
\begin{equation}\label{p3}
(q-p)(q-p)^*(q-p)=(q-p)^2.
\end{equation} From Equations (\ref{p1}),  (\ref{p2})  and (\ref{p3}), we deduce
\begin{equation*}
p^2+p^*=pp^*p+p^*=pp^*+p^* p.
\end{equation*}
Multiplying this identity by $p$ on the left and on the right, and havind in mind Equation (\ref{p2}), we get
\begin{equation*}
p^4+p^2=p^3+p^3.
\end{equation*}
Equivalently,
\begin{equation*}
p^2(\textbf{1}-p)^2=0.
\end{equation*}
From the last identity, and Equations (\ref{p2}) and (\ref{p3}) it is clear that
$$0=p^2(\textbf{1}-p)^2q^2=p^2(q-p)^2=pp^*p(q-p)(q-p)^*(q-p).$$
By cancellation, we get $p(q-p)=0$. That is $pp^*p=p^2=pq=p$, which shows that $p\in\textrm{Proj}(A)$, as claimed.
\end{proof}
\end{proposition}
Next our aim is to characterize the maximal and minimal elements on a unital C$^*$-algebra with respect to the diamond partial order.
As usual, we denote by  $A^{-1}_l$, $ A^{-1}_r$, the set of left invertible elements and right invertible elements, respectively, in a unital C$^*$-algebra $A$.

\begin{proposition} Let $A$ be a unital prime C*-algebra. The following conditions are equivalent:

\begin{enumerate}

\item $a\in A^\wedge$ and $a$ is maximal with respect to the diamond partial order,

\item $a\in A^{-1}_l\cup A^{-1}_r$.

\end{enumerate}

\end{proposition}

\begin{proof} Let $a\in A^\wedge$. It is straightforward to see that $$a\ld a+(1-aa^\dagger)x(1-a^\dagger a),$$ for every $x\in A$. If we suppose that $a$ is maximal with respect to "$\ld$", this gives $(1-aa^\dagger)x(1-a^\dagger a)=0$ for every $x\in A$. Since $A$ is prime, it yields to $\textbf{1}=aa^\dagger$ or $\textbf{1}=a^\dagger a$.

Reciprocally, assume that $a\in A^{-1}_l$. Let $b\in A$ with $a\ld b$. Then, $aa^*a=ab^*a$ and there exist $x,y\in A$ satisfying $a=bx=yb$. Being $a$ left invertible, from the first identity we get $a^*a=b^*a=a^*b$. Multiplying by $x$ on the right, we obtain $a^*ax=a^*bx=a^*a$ which, by *-cancellation, shows $ax=a$. Since $a\in A^{-1}_l$, this finally gives $x=\textbf{1}$ and, hence, $a=b$. Similar considerations can be made if we suppose $a\in A^{-1}_r$.
\end{proof}
Let $A$ be a unital C$^*$-algebra. A nonzero element $u\in A$ is said to be of \emph{rank-one} if
 $u$ belongs to some minimal left (right) ideal of $A$. Equivalently, $u$ is of rank-one if $ u \neq 0$ and
 $uAu=\CC u$. This is also equivalent to the condition
$u\neq 0$ and $|\sigma(xu)|\setminus
\{0\}\leq 1, $ for all $x\in A$, or equivalently
 $|\sigma(ux)|\setminus \{0\}\leq 1, $ for all $x\in A$.
Here and subsequently, given $a\in A$, $\sigma (a)$ denotes the spectrum of $a$ and  $\textrm{r}(a)$ its spectral radius.
By $F_1(A)$ we denote the set of all rank-one elements of $A$. It is well known that $u\in F_1(A)$ if and only if there exists a unique linear functional $\tau_u$ on $A$ such that $\tau_u(x)u=uxu$, for all $x\in A$. Moreover, $\sigma(a)=\{0,\tau_u(\uno)\}$. The complex number $\tau (u):=\tau_u(\uno)$ is called the \emph{trace} of $u$.
An element $x$ of $A$ is
\emph{finite} (\emph{compact}) in $A$, if the wedge operator
$x\wedge x:A \to A$, given by $x\wedge x (a)=xax$, is a finite rank
(compact) operator on $A$. It is known that the ideal
$ F (A)$ of finite rank elements in $A$ coincides with its \emph{socle}, $\soc (A)$, that is, the sum of all  minimal
right (equivalently left) ideals of $A$, and that $K
(A)=\overline{\soc(A)}$ is the ideal of compact elements in $A$.
Every element in the socle of a C$^*$-algebra is a linear
combination of minimal projections.  Moreover $\soc(A)\subseteq A^\wedge$.
%We refer to \cite{BMSW},  for the basic references on the socle.
\smallskip

Finally, recall that a non zero ideal $I$ of $A$ is \emph{essential} if it has non zero intersection with every non zero ideal of $A$, or equivalently (as every C$^*$-algebra is semisimple), the condition $aI=0$ for all $a\in A$, implies $a=0$.
\begin{proposition}\label{minimals}Let $A$ be a unital C$^*$-algebra with essential socle. Then $F_1(A)=\mathrm{Minimals}_{\ld} (A\setminus\{0\})$.
\end{proposition}

\begin{proof} Let us first show that for every $a\in A\setminus\{0\}$, there exists $u\in F_1(A)$ such that $u\ld a$.

Since $A$ is semisimple and has essential socle, given $a\in A\setminus\{0\}$, there exists  $w\in F_1(A)$ such that $aw\neq 0$. Let $v=w(aw)^\dagger\in F_1(A)$. Then $av$ is a minimal projection in $A$. Set $u=ava$. Clearly, $uA\subset aA$ and $Au\subset Aa$. Moreover, $$uu^*u=(ava)(ava)^*(ava)=avaa^*avava=(ava)a^*(ava)=ua^*u,$$that is, $u\ld a$.

To finish the proof, we show that, given $u,v\in F_1(A)$, with $u\ld v$ then  $u=v$. Indeed, since $u\leq_{sp} v$ (and $u,v\in F_1(A)$), it is clear that $uA=vA$ and $Au=Av$. In particular, $v=uz=wu$ for some $w,z\in A$. Accordingly, from $uu^*u=uv^*u$ we get $uu^*u=uu^*w^*u$. By *-cancellation, we get $u^*u=u^*w^*u=z^*u^*u$, and hence $u^*=z^*u^*=v^*$. That is, $u=v$.
\end{proof}

Let $A$ be a unital prime C*-algebra with non zero socle. Then $A$ is primitive and has essential socle.
Let us assume that $e$ is a minimal projection in $A$. Then
the minimal left ideal $ Ae$ can be endowed with an inner product, $\langle x, y\rangle e = y^*x$
(for all $ x, y \in Ae$), under which $Ae$ becomes a Hilbert space in the algebra norm.
Let $\rho : A \to B(Ae)$ be the left regular representation on $Ae$, given by $\rho(a)(x) = ax$ ($x\in Ae$).
The mapping  $\rho$ is an isometric irreducible $^*$-representation, satisfying:
\begin{enumerate}
\item $\rho(\soc(A))=F(Ae)$,
\item $\rho(\overline{\soc(A)})=K(Ae)$,
\item $\sigma_A(x)=\sigma_{ B(Ae)}(\rho(x))$, for every $x\in A$.
\end{enumerate}
(See \cite[Section F.4]{BMSW}.)

\begin{proposition}\label{invprdiamond}Let $A$ be a unital prime C*-algebra with non zero socle and $a\in A$. The following conditions are equivalent:

\begin{enumerate}

\item $a\in A^{-1}$,

\item For every $u\in F_1(A)$, there exist non zero $x\in uA$ and $y\in Au$ such that $x,y\ld a$.

\end{enumerate}

\end{proposition}

\begin{proof}
Notice that for every  $a\in A^{-1}$ (even though $A$ is non necessarily prime), and every $p\in \textrm{Proj}(A)$, $pa\ld a$ and $ap\ld a$. In particular, for every  $a\in A^{-1}$ and every
$u\in F_1(A)$, $uu^\dagger a \ld a$ and $au^\dagger u \ld a$. This proves that \emph{(1)}$\Rightarrow$\emph{(2)}.
%Take $a\in A^{-1}$. Given $u\in A$, we claim that $uu^\dagger a\ld a$. Indeed, $uu^\dagger a A\subset aA$ trivially, $Auu^\dagger a\subset Aa$ by the invertibility of $a$ and $$(uu^\dagger a)(uu^\dagger a)^*(uu^\dagger a)=uu^\dagger aa^*uu^\dagger a.$$

Reciprocally, assume that condition \emph{(2)} is fulfilled. For any $u\in F_1(A)$, there exist $x,y\in A$ such that $uxA\subset aA$ and $Ayu\subset Aa$. Consequently, $ux=az$ and $yu=wa$ for some $z,w\in A$. Therefore, $$u=\tau (ux(ux)^\dagger)u=ux(ux)^\dagger u=a\left(z(ux)^\dagger u\right) , $$and $$u=\tau((yu)^\dagger yu)u=u(yu)^\dagger yu=\left(u(yu)^\dagger w\right)a.$$
In particular, $\textrm{ann}_l(a)\subseteq \textrm{ann}_l(u)$ and $\textrm{ann}_r(a)\subseteq \textrm{ann}_r(u)$ , for every  $u\in F_1(A)$.
Since $A$ has essential socle, we conclude that $\textrm{ann}_l(a)=\{0\}$ and  $\textrm{ann}_r(a)=\{0\}$. That is, $a$ is not a zero divisor.
Fix $e$ a minimal projection in $A$ and let $\rho$ denote the  left regular representation on $B(Ae)$.
From $\textrm{ann}_r(a)=\{0\}$ it is clear that $\rho(a)$ is injective. Moreover, given $ze\in Ae$, by hypothesis, there exists $w\in A$ such that $ze=awze=\rho(a)(wze)$. This shows that $\rho(a)$ is surjective, and hence $\rho(a)$ is invertible.  That is, $a\in A^{-1}$.
\end{proof}
It is straightforward to show that for every unitary element $u$ in a $C^*$-algebra $A$,  $a\ld b$ if and only if $ua\ld ub$, for every $a,b\in A.$
\begin{proposition}\label{iso-diam} Let $A$ be a unital $C^*$-algebra, and $u\in A$.
\begin{enumerate}
\item If  $u^* u=\lambda\textbf{1}$,  for some $\lambda \in \RR ^+$, then
$$a\ld b \Rightarrow ua\ld ub, \quad \mbox{for every }\,a,b\in A.$$
\item  If  $u u^* =\lambda\textbf{1}$,  for some $\lambda \in \RR ^+$, then
$$a\ld b \Rightarrow au\ld bu, \quad \mbox{for every }\,a,b\in A.$$
\end{enumerate}
\end{proposition}
\begin{proof}
We only prove the first assertion (the second can be shown in a similar way).
Suppose that $u^* u=\lambda\textbf{1}$, and let $a,b \in A$ with $a\ld b $.
As $aA\subseteq bA$ obviously $uaA\subseteq ubA,$ and since $u$ is left invertible and $Aa\subseteq Ab$, we get $Aua\subseteq Aub.$ Moreover,
$$(ua)(ua)^*(ua)=uaa^*u^*ua=\lambda uaa^*a=\lambda uab^*a=(ua)(ub)^*(ua),$$ which shows that $ua\ld ub$.
\end{proof}
We conclude this section by characterizing the scalar multiples of isometries and co-isometries in a  unital prime C$^*$-algebra with non zero socle.
\begin{proposition}\label{iso-diamond-prime}Let $A$ be a unital prime C*-algebra with non zero socle and $u\in A^\wedge$.

\begin{enumerate}

\item  The condition  $$a\ld b \Leftrightarrow au\ld bu, \quad \mbox{for every }\,a,b\in A,$$ implies that $uu^* =\lambda\textbf{1}$,  with $\lambda \in \RR ^+$.

\item  The condition  $$a\ld b \Leftrightarrow ua\ld ub,\quad\mbox{ for every }\,a,b\in A,$$ implies that $u^*u =\lambda\textbf{1}$,  with $\lambda\in \RR ^+$.

\end{enumerate}

\end{proposition}

\begin{proof} As in the previous proposition we only need to prove the first assertion.
Assume that
\begin{equation}\label{ispr}
a\ld b \Leftrightarrow au\ld bu, \quad \mbox{for every }\,a,b\in A.
\end{equation}
It is clear that $\textrm{ann}_l (u)=\{0\}.$ Since $u\in A^\wedge$, we conclude that $u$ is right invertible, that is, $uu^\dagger=\textbf{1}.$

Notice that, $$u^\dagger p \ld u^\dagger,\quad \mbox{ for every }\, p\in \textrm{Proj}(A).$$
Indeed, let $p\in \textrm{Proj}(A).$ Then $p\ld \textbf{1}$. It is clear that $u^\dagger p A \subseteq u^\dagger A$ and since  $u^\dagger$ is left invertible $Au^\dagger p\subseteq Au^\dagger.$ Finally $$(u^\dagger p)(u^\dagger p)^* (u^\dagger p)=u^\dagger p (u^\dagger)^*u^\dagger p,$$ gives $u^\dagger p \ld u^\dagger$.
In the same way,
 $$u^\dagger -u^\dagger p= u^\dagger (\textbf{1}-p) \ld u^\dagger,\quad \mbox{ for every }\, p\in \textrm{Proj}(A).$$
Let us apply the condition (\ref{ispr}) with $a=u^\dagger p$ and $b=u^\dagger$. Therefore,
\begin{equation}\label{udg}
u^\dagger p u \ld u^\dagger u.
\end{equation}
Applying now the condition (\ref{ispr}) with $a=u^\dagger -u^\dagger p$ and $b=u^\dagger$, we obtain
\begin{equation}\label{udg2}
u^\dagger u -u^\dagger p u \ld u^\dagger u.
\end{equation}
Having in mind Proposition \ref{proj} and Equations (\ref{udg}) and (\ref{udg2}), we conclude that $u^\dagger p u \in \textrm{Proj}(A)$, for every $p\in \textrm{Proj}(A)$. That is,
$$u^\dagger p u =u^*p (u^\dagger)^*, $$ for every $p\in \textrm{Proj}(A)$. Multiplying this last identity by $u$ on the left, and by $u^*$ on the right, we deduce that
$$puu^*=uu^*p,\quad \mbox{ for every }\, p\in \textrm{Proj}(A).$$
In particular, $uu^*$ commutes with every minimal projection, and hence
$$xuu^*=uu^*x,\quad \mbox{ for every }\, x\in \soc(A).$$
Being $\soc(A)$ essential, $uu^*$ lies in the center of $A$, $\textrm{Z}(A)$. As $A$ is prime, $\textrm{Z}(A)=\CC \textbf{1},$ that is, $uu^* =\lambda\textbf{1}$,  for some $\lambda\in \RR ^+$.
\end{proof}
\begin{remark} Notice that the same conclusions hold when $A$ is a unital $C^*$-algebra with trivial center and either $A$ is linearly spanned by its projections, or $A$ has real rank zero (that is, the set of all real linear
combinations of orthogonal projections is dense in the set of all hermitian
elements of $A$, \cite{BrownPed91}).
\end{remark}
\section{Maps preserving the diamond partial order}
Let $A$ and $B$ be C$^*$-algebras. Recall that a  linear map $T:A \to B$ is a \emph{Jordan homomorphism} if $T(a^2)=T(a)^2$, for all $a\in A$, or equivalently, $T(ab+ba)=T(a)T(b)+T(b)T(a)$ for every $a,b$ in $A$. A bijective Jordan homomorphism is named \emph{Jordan isomorphism}. Clearly every homomorphism and every anti-homomorphism is a Jordan homomorphism.  A well known result of Herstein, \cite{Her56}, states that every surjective Jordan homomorphism $T:A \to B$ is either an homomorphism or an
anti-homomorphism whenever $B$ is prime.

Recall also that if $T:A \to B$ is a Jordan homomorphism then \begin{equation}\label{triple} T(abc+cba)=T(a)T(b)T(c)+T(c)T(b)T(a),\end{equation} for all $a,b,c\in A$.

The mapping $T$ is called \emph{selfadjoint} if $T(a^*)=T(a)^*$, for every $a\in A$. Selfadjoint Jordan homomorphisms are called \emph{Jordan $^*$-homomorphisms}.
It can be easily checked that every $^*$-homomorphism and every $^*$-anti-homomorphism preserves the diamond partial order. Hence, it is also the case for  every  Jordan $^*$-homomorphism $T:A\to B$ onto a prime $C^*$-algebra.

In the next proposition we show that every Jordan $^*$-homomorphism preserves the diamond partial order in the setting of all regular elements. It can be proved by using \cite[Remark 8]{BMM12} and \cite[Proposition 3.1]{BuMaMo-pr15}. However we include its proof here for the sake of completeness.

\begin{proposition}\label{jordiam} Let $A$ and $B$ be C$^*$-algebras. If $T:A\to B$ is a Jordan $^*$-homomorphism, then
$$ a\ld b\quad \mbox{implies}\quad T(a)\ld T(b), \quad \mbox{for all }\, a,b\in A^\dag.$$
\end{proposition}
\begin{proof}
First notice that every Jordan $^*$-homomorphism $T:A\to B$ between C$^*$-algebras strongly preserves Moore-Penrose invertibility, that is, $T(a^\dagger)=T(a)^\dagger$ for every $a\in A^\wedge$. Indeed if $a\in A^\wedge$ and $b=a^\dagger$, from Equation~(\ref{triple}) it is clear that $T(a)=T(a)T(b)T(a)$ and $T(b)=T(b)T(a)T(b)$, that is, $T(b)$ is a generalized inverse of $T(a)$. By the uniqueness of the Moore-Penrose inverse, it remains to show that $T(b)T(a)$ and $T(a)T(b)$ are selfadjoint. As $a=b^*a^*a=aa^*b^*$, in particular $2a=b^*a^*a+aa^*b^*$, and since $T$ is a Jordan $^*$-homomorphism, it is clear that
$$2T(a)=T(b)^*T(a)^*T(a)+T(a)T(a)^*T(b)^*.$$ Multiplying on the left by $T(a)^*$, we get  that $$T(a)^*T(a)=T(a)^*T(a)T(a)^*T(b)^*,$$ or equivalently $T(a)^*T(a)(T(b)T(a)-T(a)^*T(b)^*)=0$, which implies that $T(a)=T(a)T(a)^*T(b)^*$, and hence $T(b)T(a)=T(b)T(a)T(a)^*T(b)^*$ is selfadjoint.

Now, we claim that $T:A\to B$ preserves the minus partial order. Let $a,b\in A^\wedge$. We know that $a\leq^-b$ if and only if there exists a generalized inverse $b^-$ of $b$ such that $a=ab^-a=ab^-b=bb^-a$. From Equation (\ref{triple}), as $a=ab^-a$ and $2a=ab^-b+bb^-a$, we have $$T(a)=T(a)T(b)^-T(a)\quad \mbox{ and}\quad 2T(a)=T(a)T(b)^-T(b)+T(b)T(b)^-T(a).$$  Multiplying  the last identity by $T(b)^-T(a)$ on the right, and havind in mind that  $T(b)^-$ is a generalized inverse of $T(b)$, we get \begin{eqnarray*}2T(a)&=& T(a)T(b)^-T(b)T(b)^-T(a)+T(b)T(b)^-T(a)T(b)^-T(a)\\ &=& T(a)+T(b)T(b)^-T(a).
\end{eqnarray*} Consequently, $T(a)=T(b)T(b)^-T(a)$. Similarly, it can be proved that $T(a)=T(a)T(b)^-T(b)$, which yields to $T(a)\leq^{-} T(b)$.

We conclude the proof by applying Proposition \ref{diamprop} \emph{(3)}.
\end{proof}
In this section we wonder whether Jordan $^*$-homomorphisms arise from linear maps preserving the diamond partial order.

The study of linear maps between C$^*$-algebras preserving the star partial order (and its generalization (R2)), was connected  in \cite{BuMaPa15} with that of orthogonality preserves (\cite{Orth08}).
Recall that two elements $a,b$ in a $C^*$-algebra  $A$ are called \emph{orthogonal} (denoted by $a\perp b$) if $ab^*=b^*a=0$. Given $a,b$ in a C$^*$-algebra $A$, it is straightforward that $a\leq_{*}  (a+b)$ if and only if $a\perp b$.
From Proposition \ref{diamprop} \emph{(2)} it follows that for a regular element $a\in A$, if $a\perp b, $ then $ a\ld (a+b)$. The following example shows that the reciprocal does not hold. Hence we cannot expect to apply the same orthogonality arguments in order to describe linear maps between $C^*$-algebras preserving the diamond partial order.
\begin{example} Let $A= M_2(\mathbb{C})$ and $$a=\left( \begin{array}{ccc}
1 & 0 \\
0 & 0 \end{array} \right),\qquad u=\left( \begin{array}{ccc}
0 & 1/\sqrt{2} \\
0 & 1/\sqrt{2} \end{array} \right).$$
It is clear that $a$ is a projection and $u$ is a partial isometry in $A$. It can be checked that
$$a u^* a =0,$$
$$a A = \left\{ \left( \begin{array}{ccc}
x & y \\
0 & 0 \end{array} \right): x,y\in \mathbb{C} \right\} \subseteq (a+u) A =  \left\{ \left( \begin{array}{ccc}
x+z/\sqrt{2}& y+t/\sqrt{2} \\
z/\sqrt{2} & t/\sqrt{2} \end{array} \right) : x,y,z,t\in \mathbb{C}\right \}.$$  Similarly,
$$A a= \left\{ \left( \begin{array}{ccc}
x & 0 \\
z & 0 \end{array} \right): x,z\in \mathbb{C}\}\subseteq  A (a+u)=  \{ \left( \begin{array}{ccc}
x& (x+y)/\sqrt{2} \\
z & (z+t)/\sqrt{2} \end{array} \right) : x,y,z,t\in \mathbb{C}\right\}.$$
This shows that $a\ld(a+u)$. However $a$ and $u$ are not orthogonal since $u^*a\neq 0$.
\end{example}

Our first main result partially uses similar arguments to those of \cite[Theorem 3.2, Theorem 3.6]{BuMaMo-pr15}.
\begin{theorem}\label{main1} Let $A$ and $B$ be unital $C^*$-algebras with essential socle. Assume that $B$ is prime. Let $T:A\to B$ be a surjective linear map and $h=T(\textbf{1})$.
The following conditions are equivalent:
\begin{enumerate}
\item $a\ld b \Leftrightarrow T(a)\ld T(b),$  for every $a,b\in A$,
\item $hh^*=h^*h=\lambda\textbf{1},$ with $\lambda \in \RR^+$, and $T=hS$, where $S:A\to B$ is either a $^*$-isomorphism or a $^*$-anti-isomorphism.
\end{enumerate}
\end{theorem}
\begin{proof}

We only need to prove that \emph{(1)}$\Rightarrow$\emph{(2)}, since the converse is straightforward.
Suppose then that
\[a\ld b \Leftrightarrow T(a)\ld T(b), \quad \text{ for every }a,b\in A.\]
Notice that $T$ is injective: if $T(x)=0$, then $T(x)\ld T(0)$, which by assumption, gives that $x\ld 0$, and finally $x=0$.

We claim that $T(F_1(A))=F_1(B)$. Indeed, pick $u\in F_1(A)$. From Proposition \ref{minimals}  there exists $T(v)\in F_1(B)$ such that $T(v)\ld T(u)$.  By hypothesis we have $v\ld u$. As $u\in F_1(A)$, and $v\neq 0$, Proposition \ref{minimals} implies that $v=u$. That is, $T(v)=T(u)$, which shows that $T(u)\in F_1(B)$. The same arguments applied to $T^{-1}$ gives $T(F_1(A))=F_1(B)$.
It can be shown that the maximal linear subspaces of $\soc(A)$ consisting of elements of rank at most one are either of the form $uA$ or $Au$, for some $u\in F_1(A)$ (see \cite[Lemma 2.18]{BuMaMo-pr15}). Therefore, $T(uA), T(Au)\in \{T(u)B,BT(u)\}$, for every $u\in F_1(A)$.

Next we prove that $T$ preserves invertibility.  For this purpose, take $a\in A^{-1}$. Given $T(u)\in F_1(B)$, by Proposition \ref{invprdiamond}, there exist non zero elements $x_0\in uA$ and $y_0\in Au$, such that $x_0,y_0 \ld a$.  If $T^{-1}(T(u)B)=uA$, take $x=x_0$. Otherwise, take $x=y_0$. Then $T(x)\in T(u)B$ and $T(x)\ld T(a)$.
Similarly, we find $T(y)\in BT(u)$ with $T(y)\ld T(a)$. By Proposition \ref{invprdiamond}, $T(a)\in B^{-1}$.  In particular $h=T(\textbf{1})\in B^{-1}$.

Let us define the linear mapping $S:A\to B$ as $S(x)=h^{-1}T(x)$ for every $x\in A$. It is clear that $S$ is unital, bijective and preserves invertibility. By \cite[Theorem 1.1]{BrFoSe03},  $S$ is a Jordan isomorphism. Since $B$ is prime, we known that $S$ is either an isomorphism or an anti-isomorphism. We may  assume, without loss of generality, that $S$ is an isomorphism. Then
$$T(xy)=T(x)h^{-1}T(y),\quad \mbox{for all}\, x,y \in A.$$
Let $u$ be a unitary element in $A$. It is clear that
$au\ld bu$ if and only if $a\ld b$.
By hypothesis,
$$T(a)\ld T(b)\Leftrightarrow T(au)\ld T(bu)\Leftrightarrow T(a)h^{-1}T(u)\ld T(b)h^{-1}T(u).$$
Taking into account Proposition \ref{iso-diamond-prime}, we conclude that $S(u)S(u)^* =\lambda\textbf{1}$,  with $\lambda \in \RR ^+$.
As $S(u)\in B^{-1}$, it follows that $S(u)S(u)^* =S(u)^*S(u)=\lambda\textbf{1}$ . In particular, $S(u)$ is normal, for every unitary element $u\in A$.
Consequently, as $S$ is a unital Jordan homomorphism, it follows that $$||S(u)||=r(S(u))=r(u)=1,\quad \mbox{for every unitary element } \, u\in A.$$
This shows that $S$ is selfadjoint (see\cite[Corollary 2]{RuDye}).

Finally, $T(x)=hS(x)$, for every $x\in A$, where $S$ is either a $^*$-isomorphism or a $^*$-anti-isomorphism. Since $T$ and $S$ both preserve the diamond partial order, we have
$$T(a)\ld T(b)\Leftrightarrow a\ld b \Leftrightarrow S(a)\ld S(b) \Leftrightarrow h^{-1}T(a)\ld h^{-1}T(b).$$
By Proposition \ref{iso-diamond-prime}, $h^{-1}$ is a scalar multiple of an isometry and, hence, $h$ is a scalar multiple of a unitary element.
\end{proof}

The next corollary can be obtained directly from Theorem \ref{main1} and the well known structure of surjective linear isometries of $B(H)$.
\begin{corollary} Let $H$ be a complex Hilbert space. If $\Phi:B(H)\to B(H)$ is a surjective linear map that preserves the diamond partial order in both directions, then there are unitary operators $U,V$ on $H$, and $\lambda\in \RR ^{+}$, such that $\Phi$ is either of the form
$$\Phi(A)=\lambda UAV\qquad \mbox{for all }\, A\in B(H),$$
or of the form
$$\Phi(A)=\lambda UA^{tr}V\qquad \mbox{for all }\, A\in B(H).$$
\end{corollary}

Recall that a $C^*$-algebra $A$ has \emph{real rank zero} if the set of all real linear
combinations of orthogonal projections is dense in the set of all hermitian
elements of $A$ (\cite{BrownPed91}). Every von
Neumann algebra and in particular the algebra  $B(H)$ of all bounded linear
operators on a complex Hilbert space $H$, has real rank zero. Other examples
of this kind of algebra include Bunce-Deddens algebras, Cuntz algebras,
AF-algebras and irrational rotation algebras.

The following observation has become a standard tool in the study of Jordan $^*$-homomorphisms: Let $A$ be a real rank zero C*-algebra, $B$ be a $C^*$-algebra and $T: A\to B$ be a bounded linear mapping sending projections to projections. Then $T$ is  *-Jordan homomorphism.

In the next theorem we consider linear maps preserving the diamond partial order on a real rank zero $C^*$-algebra under few additional conditions involving the image of the identity.

\begin{theorem}\label{d-rro} Let $A$ and $B$ be unital $C^*$-algebras. Assume that $A$ has real rank zero. Let $T:A\to B$ be a bounded linear map satisfying that
$$a\ld b\quad\mbox{ implies}\quad T(a)\ld T(b),\quad\mbox{for all }a,b\in A^\wedge.$$
The following assertions hold.
\begin{enumerate}
\item If $T(\textbf{1})\in \textrm{Proj}(B)$ then $T$ is a Jordan $^*$-homomorphism.
\item If $T(A)\cap B^{-1}$ and $T(\textbf{1})$ is a partial isometry then $T$ is a Jordan $^*$-homomorphism multiplied by a unitary element.
\end{enumerate}
\end{theorem}
\begin{proof}

Notice that  $p\ld \textbf{1}$ and $ \textbf{1}-p\ld  \textbf{1}$, for every $p\in  \textrm{Proj}(A)$. Therefore,
\begin{equation}\label{pro}
T(p)\ld T(\textbf{1})\quad\text{ and } T(\textbf{1})-T(p)\ld  T(\textbf{1}), \quad \text{for every }\, p\in  \textrm{Proj}(A).\end{equation}

In order to prove \emph{(1)} assume that $T(\textbf{1})\in \textrm{Proj}(B)$. Proposition \ref{proj} and (\ref{pro}) allow us to conclude that $T(p)\in \textrm{Proj}(B)$, for every $p\in  \textrm{Proj}(A)$. Therefore $T$ is a Jordan $^*$-homomorphism.

Now assume that $T(A)\cap B^{-1}$ and that $T(\textbf{1})$ is a partial isometry. Since $T(p)\ld T(\textbf{1})$, in particular, $T(p)\in T(\textbf{1})B\cap B T(\textbf{1})$  for every $p\in  \textrm{Proj}(A)$. Moreover, $T(\textbf{1})B$ and $B T(\textbf{1})$ are closed in view of the regularity of $T(\textbf{1})$ .  As $T$ is linear and bounded, and every
self-adjoint element in $A$ can be approximated by linear combinations of mutually orthogonal projections, we conclude that  $T(A)\subseteq T(\textbf{1})B\cap B T(\textbf{1})$. This fact together with $T(A)\cap B^{-1}$, imply that $ T(\textbf{1})\in B^{-1}$, and therefore, $ T(\textbf{1})$ is unitary. Let $S:A\to B$ be the linear mapping given by $S(x)= T(\textbf{1})^*T(x)$, for all $x\in A$. Hence $T(x)= T(\textbf{1})S(x)$, for all $x\in A$. Taking into account that $T$ preserves the diamond partial order and $ T(\textbf{1})$ is unitary, it is clear that $S$ is a  unital, bounded, linear mapping preserving the diamond partial order. As consequence, $S$ preserves projections and hence it is a Jordan $^*$-homomorphism.
\end{proof}


\begin{thebibliography}{0}

%\bibitem{Aup00} B. Aupetit, \emph{Spectrum-preserving linear mappings between Banach algebras or Jordan-Banach algebras}, J. London Math. Soc. \textbf{62} (2000) 917-924.

\bibitem{BakHa90}  J. K. Baksalary and J. Hauke, \emph{A further algebraic version of Cochran's theorem and matrix partial orderings}, Linear Algebra Appl. \textbf{127} (1990), 157-169.

\bibitem{BakMi91} J. K. Baksalary and  S. M. Mitra, \emph{Left-star and right-star partial orderings}, Linear Algebra Appl. \textbf{149} (1991), 73-89.


\bibitem{BMSW} B.A. Barnes, G. J. Murphy, M. R. F. Smyth and  T. T. West, \emph{Riesz and Fredholm theory in Banach algebras}. London. Pitman, 1982.

\bibitem{BrFoSe03} M. Bre\v{s}ar, A. Fo\v{s}ner and  P. \v{S}emrl, \emph{ A note on invertibility preservers on Banach algebras}. Proc. Amer. Math. Soc. \textbf{131} (2003), 3833-3837.

\bibitem{BrownPed91} L. G. Brown and G. K. Pedersen, \emph{C*-algebras of real rank zero}, J. Funct. Anal. \textbf{99} (1991), 131-149.


\bibitem{BMM12} M. Burgos, A. C. M\'{a}rquez-Garc\'{i}a and  A. Morales-Campoy, \emph{Linear maps strongly preserving Moore-Penrose invertibility}, Operators and Matrices. \textbf{6} (2012), 819-831.

\bibitem{Orth08} M. Burgos, F. J. Fern\'{a}ndez-Polo, J. J. Garces, J. Martinez Moreno and  A.M. Peralta, \emph{Orthogonality preservers in C*-algebras, JB*-algebras and JB*-triples}, J. Math. Anal. Appl. \textbf{348} (2008), 220-233.

\bibitem{BuMaPa15} M. Burgos, A. C. M\'{a}rquez-Garc\'{i}a and  P. Patr\'icio, \emph{On mappings preserving the sharp and star orders}, Linear Algebra Appl. \textbf{483} (2015), 268-292.

\bibitem{BuMaMo-pr15} M. Burgos, A. C. M\'{a}rquez-Garc\'{i}a  A. Morales-Campoy, \emph{Minus partial order and linear preservers}, submitted.


%\bibitem{Burgos13} M. Burgos, \emph{Orthogonality preserving linear maps on C*-algebras with non-zero socles}, J. Math. Anal. Appl. \textbf{401} (2013) 479-487.


%\bibitem{Orth09} M. Burgos, F.J. Fern\'{a}ndez-Polo, J.J. Garc\'es, A.M. Peralta, \emph{Orthogonality preservers revisited}, Asian-Eur. J. Math. \textbf{2} (2009) 387-405.

%\bibitem{Orth11} M. Burgos, J. Garc\'{e}s, A. Peralta, \emph{Automatic continuity of biorthogonality preservers between compact C*-algebras and von Neumann algebras}, J. Math. Anal. Appl. \textbf{376} (2011) 221-230.

%\bibitem{BurSan} M. Burgos, J. S\'{a}nchez-Ortega, \emph{On maps preserving zero products}, Lin. Mult. Alg. \textbf{61} (3) (2013) 323-335.

\bibitem{DjoRaMa13} D. S. Djordjevic,  D. S. Rakic and J. Marovt  \emph{Minus partial order in Rickart rings}, IMFM, Preprint series, \textbf{51} (2013), 1191.

\bibitem{DoGuma14} G. Dolinar, A. Guterman and J. Marovt, \emph{Monotone transformations on $B(H)$ with respect to the left-star and the right-star partial order}, Math. Ineq. Appl. \textbf{17} (2) (2014), 573-589.

\bibitem{Drazin78} M. P. Drazin, \emph{Natural structures on semigroups with involution}, Bull. Amer. Nath. Soc. \textbf{84} (1978), 139-141.

\bibitem{Gut07} A. E. Guterman, \emph{Monotone additive transformations on matrices}, Mat. Zametki \textbf{81} (2007),  681-692.

\bibitem{GuLiSe} A. Guterman, C.-K. Li and  P. \v{S}emrl. \emph{ Some general techniques on linear preserver problems}. Linear Algebra and its Applications , \textbf{315} (2000), 61-81.

\bibitem{HarMb92}R. Harte and  M. Mbekhta, \emph{On generalized inverses in C*-algebras}, Studia Math. \textbf{103} (1992), 71-77.

%\bibitem{HarMb93} R. Harte, M. Mbekhta, \emph{Generalized inverses in C*-algebras II}, Studia Math. \textbf{106} (1993) 129-138.

\bibitem{Hart80} R. E. Hartwig, \emph{How to partially order regular elements}, Math. Japon. \textbf{25} (1980), 1-13.

\bibitem{HarSty86} R. E. Hartwig and G. P. H. Styan, \emph{ On some characterizations of the ``star" partial ordering for matrices and rank substractivity}, Linear Algebra Appl. \textbf{82} (1986), 145-161.

\bibitem{Her56} I. N. Herstein, \emph{Jordan homomorphisms}.  Trans. Amer. Math. Soc. \textbf{81}  (1956), 331-341.

%\bibitem{Kovacs05}I. Kovacs, \emph{Invertibility-preserving maps of C*-algebras with real rank zero}, Abstr. Appl. Anal. Volume 2005, Number \textbf{6} (2005) 685-689.

\bibitem{LebPaTh13} L. Lebtahi, P. Patricio and  N. Thome, \emph{The diamond partial order in rings}, Lin. Mult. Alg., \textbf{62} (3) (2014), 386-395.

%\bibitem{MaRaDj14} J. Marovt, D. S. Raki\'{c}, D. S. Djordjevi\'{c}, \emph{Star, Left-star, and right-star partial orders in Rickart *-rings}, Lin. Mult. Alg., \textbf{63(2)} (2015) 343-365.

%\bibitem{Marovt15} J. Marovt, \emph{On partial orders in Rickart rings}, Lin. Mult. Alg., \textbf{63} (9) (2015) 1707-1723.

%\bibitem{MaRaDj15} J. Marovt,  D. Raki\'c, D. Djordjevi\'c, \emph{Star, left-star, and right-star partial orders in Rickart*-rings}, Lin. Mult. Alg., \textbf{63} (2) (2015) 343-365.
\bibitem{Molnar} L. Moln\`ar,  Selected Preserver Problems on Algebraic Structures of Linear Operators and on Function Spaces, Series: Lecture Notes in Mathematics 1895 (2007).

\bibitem{Mi87} S. K. Mitra, \emph{On Group Inverses and the Sharp Order}, Linear Algebra Appl. \textbf{92} (1987), 17-37.

\bibitem{Mi91} S.K. Mitra, \emph{Matrix partial order through generalized inverses: unified theory}, Linear Algebra Appl. \textbf{148} (1991), 237-263.

\bibitem{MiBhiMa} S. K. Mitra, P. Bhimasankaram and  S. B. Malik, Matrix partiar orders, shorted operators and applications. Hackensack, NJ World Scientific Publishing Company (2010).

\bibitem{RuDye} B. Russo and  H.A. Dye,\emph{A note on unitary operators in C$^*$-algebras}, Duke Math. J. \textbf{33} (1966), 413-416.

\bibitem{Semrl10} P. \v{S}emrl, \emph{Automorphisms of B(H) with respect to minus partial order}, J. Math. Anal. Appl. \textbf{369} (1) (2010), 205-213.

\end{thebibliography}
\end{document}